\documentclass[12pt,a4paper]{article}
\usepackage{amsmath,amssymb,amsthm,enumerate,mathrsfs}
\usepackage[numbers]{natbib}

\def\A{{\mathcal{A}}} 
\def\B{{\mathcal{B}}}

\def\F{{\mathcal{F}}} 
\def\G{{\mathcal{G}}} 
\def\I{{\mathcal{I}}}

\def\M{{\mathcal{M}}}  
  
\def\P{{\mathcal{P}}}  
\def\R{{\mathbb{R}}}
\def\U{{\mathcal{U}}} 

\theoremstyle{plain}
\newtheorem{thm}{Theorem}[section]
\newtheorem{lem}{Lemma}[section]

\newtheorem{pro}{Proposition}[section]

\theoremstyle{definition}
\newtheorem{df}{Definition}[section]
\newtheorem{ex}{Example}[section]

\newtheorem{rem}{Remark}[section]

\numberwithin{equation}{section}
\allowdisplaybreaks

\title{Maharam-Types and Lyapunov's Theorem for Vector Measures on Locally Convex Spaces without Control Measures\thanks{This version of the paper was completed during Sagara's visit to the Johns Hopkins University. This research is supported by a Grant-\hspace{0pt}in-\hspace{0pt}Aid for Scientific Research (No.\,26380246) from the Ministry of Education, Culture, Sports, Science and Technology, Japan.}}

\date{\today}
\author{M. Ali Khan \\
{\small Department of Economics, The Johns Hopkins University} \\[-4pt]
{\small Baltimore, MD 21218, United States} \\[-4pt]
{\footnotesize e-mail: akhan@jhu.edu}
\and
\\
Nobusumi Sagara\thanks{Corresponding author.} \\
{\small Department of Economics, Hosei University} \\[-4pt]
{\small 4342, Aihara, Machida, Tokyo 194--0298, Japan} \\[-4pt]
{\footnotesize e-mail: nsagara@hosei.ac.jp}}

\begin{document}
\begin{titlepage}
\maketitle
\setcounter{page}{0}
\thispagestyle{empty}
\end{titlepage}
\clearpage

\begin{abstract}
 We formulate the saturation property for vector measures in lcHs as a nonseparability condition on the derived Boolean $\sigma$-\hspace{0pt}algebras by drawing on the topological structure of vector measure algebras. We exploit a Pettis-\hspace{0pt}like notion of vector integration in lcHs, the Bourbaki--\hspace{0pt}Kluv\'anek--\hspace{0pt}Lewis integral, to derive an exact version of the Lyapunov convexity theorem in lcHs without the BDS property. We apply our Lyapunov convexity theorem to the bang-\hspace{0pt}bang principle in Lyapunov control systems in lcHs to provide a further characterization of the saturation property.   

\bigskip

\noindent
{\bfseries Key words:} vector measure; locally convex Hausdorff space; Maharam-\hspace{0pt}type; saturation; Lyapunov convexity theorem; multifunction; control system; bang-\hspace{0pt}bang principle.

\smallskip

\end{abstract}

\tableofcontents
\clearpage

\section{Introduction}
Lyapunov's convexity theorem on the range of an atomless vector measure has proved to be of tremendous use and significance in applied work in the theories of optimal control, of statistical decisions, of Nash equilibria in large games, and of Walrasian equilibria in  mathematical economics; see the references of the extended introduction in \cite{ks14}. As explored in \cite{ks13,ks14,ks15} at length, the so-\hspace{0pt}called saturation property is an indispensable structure on measure spaces for the Lyapunov convexity theorem to be valid without the closure operation for measures taking values in infinite-\hspace{0pt}dimensional spaces. This property of  measure spaces is by now a well established notion in measure theory, as formulated in the scalar measure case (see \cite{fr12,hk84,ma42}). For vector measures with values in locally convex Hausdorff spaces (lcHs), the notion of saturation is easily adapted under the Bartle--\hspace{0pt}Dunford--\hspace{0pt}Schwartz (BDS) property of lcHs, a property  automatically satisfied in Banach spaces (see \cite{ks13,ks15}), to guarantee the existence of control measures for any vector measure.  

A natural question pertains to a fruitful  formulation of the notion of saturation for vector measures in lcHs {\it without} control measures. Such a motivation stems from the recent development on the Lyapunov convexity theorem in Banach spaces, and we provide an answer to this question here. The main contribution of this answer  hinges on what we see as the following threefold contribution:
\begin{enumerate}[(i)]
\item We formulate the saturation property for vector measures in lcHs as a nonseparability condition on the Boolean $\sigma$-\hspace{0pt}algebras along the lines of \cite{fr12,ks09,ma42,ri92} by drawing on the topological structure of vector measure algebras established in the important monograph \cite{kk75}.  

\item We exploit a Pettis-\hspace{0pt}like notion of vector integration in lcHs, the Bourbaki--\hspace{0pt}Kluv\'anek--\hspace{0pt}Lewis integral, (see \cite{bo59,kl70,le70}) to derive an exact version of the Lyapunov convexity theorem in lcHs without the BDS property, a result  that is a natural extension of those presented in \cite{ks13,ks15} (necessity and sufficiency). 

\item We apply our Lyapunov convexity theorem to the bang-\hspace{0pt}bang principle established in \cite{ks14} and Lyapunov control systems in lcHs explored in \cite{kk75,kn75} to provide a further characterization of saturation.   
\end{enumerate}

The organization of the paper is as follows. In Section 2 we introduce the notion of Maharam types and saturation in Boolean $\sigma$-\hspace{0pt}algebras and define vector measure algebras in lcHs. Section 3 deals with the space of integrable functions with respect to a vector measure in lcHs and investigates its topological properties, especially  completeness and separability. Section 4 presents the main result of the paper in which the equivalence of the saturation property of a vector measure algebra and the Lyapunov convexity theorem is established in full generality. Section 5 contains the application of the main result to the bang-\hspace{0pt}bang principle and Lyapunov control systems.

\section{Preliminaries}
\subsection{Boolean Algebras and Maharam Types}
Let $\F$ be a Boolean algebra with binary operations $\vee$ and $\wedge$, and a unary operation $^c$, endowed with the order $\le$ given by $A\le B\Longleftrightarrow A=A\wedge B$, where $\O=\Omega^c$ is the smallest element in $\F$ and $\Omega=\O^c$ is the largest element in $\F$. A subset $\I$ of $\F$ is an \textit{ideal} if $\O\in \I$, $A\vee B\in \I$ for every $A,B\in \I$ and $B\le A$ with $A\in \I$ implies $B\in \I$. The \textit{principal ideal} $\F_E$ generated by $E\in \F$ is an ideal of $\F$ given by $\F_E=\{ A\in \F\mid A\le E \}$, which is a Boolean algebra with unit $E$. An element $A\in \F$ is an \textit{atom} if $A\ne \O$ and $E\le A$ with $E\in \F$ implies either $E=\O$ or $E=A$; $\F$ is \textit{nonatomic} if it has no atom. 

A \textit{subalgebra} of $\F$ is a subset of $\F$ that contains $\Omega$ and is closed under the Boolean operations $\vee$, $\wedge$ and ${}^c$. A subalgebra $\U$ of $\F$ is \textit{order-\hspace{0pt}closed} with respect to the order $\le$ if any nonempty upwards directed subsets of $\U$ with its supremum in $\F$ has the supremum in $\U$. A subset $\U\subset \F$ \textit{completely generates} $\F$ if the smallest order closed subalgebra in $\F$ containing $\U$ is $\F$ itself. The \textit{Maharam type} of a Boolean algebra $\F$ is the smallest cardinal of any subset $\U\subset \F$ which completely generates $\F$. By $\kappa(\F)$ we denote the Maharam type of $\F$. A Boolean algebra is \textit{saturated} if for every $E\in \F$ with $E\ne \O$ the Maharam type $\kappa(\F_E)$ of the principal ideal $\F_E$ generated by $E$ is uncountable. A Boolean algebra $\F$ is nonatomic if and only if $\kappa(\F_E)$ is infinite for every $E\in \F$ with $E\ne \O$ (see \cite[Proposition 2.1]{ks14}).

\subsection{Vector Measure Algebras in lcHs}
Let $(\Omega,\F)$ be a measurable space and $X$ be a locally convex Hausdorff space (briefly, lcHs). A set function $m:\F\to X$ is \textit{countably additive} if for every pairwise disjoint sequence $\{ A_n \}$ in $\F$, we have $m(\bigcup_{n=1}^\infty A_n)=\sum_{n=1}^\infty m(A_n)$, where the series is unconditionally convergent with respect to the locally convex topology on $X$. It is well known that if $m$ is countably additive with respect to ``some" locally convex topology on $X$ that is consistent with the dual pair $\langle X,X^* \rangle$, then it is countably additive with respect to ``any" locally consistent convex topology on $X$ (see \cite[Proposition 4]{tw70}). This is a consequence of Orlicz--\hspace{0pt}Pettis theorem (see \cite[Theorem 1]{mc67}). Therefore, the countable additivity of vector measures is independent of the particular topologies lying between the weak and Mackey topologies of $X$. 

For a vector measure $m:\F\to X$, a set $N\in \F$ is \textit{$m$-\hspace{0pt}null} if $m(A\cap N)=\bold{0}$ for every $A\in \F$. An equivalence relation $\sim$ on $\F$ is given by $A\sim B$ if and only if $A\triangle B$ is $m$-\hspace{0pt}null, where $A\triangle B$ is the symmetric difference of $A$ and $B$ in $\F$. The collection of equivalence classes is denoted by $\widehat{\F}=\F/\sim$ and its generic element $\widehat{A}$ is the equivalence class of $A\in \F$. The lattice operations $\vee$ and $\wedge$ in $\widehat{\F}$ are given in a usual way by $\widehat{A}\vee \widehat{B}=\widehat{A\cup B}$ and $\widehat{A}\wedge \widehat{B}=\widehat{A\cap B}$. The unary operation ${}^c$ in $\widehat{\F}$ is obtained for taking complements in $\widehat{\F}$ by $\widehat{A}^{c}=\widehat{(A^c)}$. Under these operations $\widehat{\F}$ is a Boolean $\sigma$-\hspace{0pt}algebra. Let $\hat{m}:\widehat{\F}\to X$ be an $X$-\hspace{0pt}valued countably additive function on $\widehat{\F}$ defined by $\hat{m}(\widehat{A})=m(A)$ for $\widehat{A}\in \widehat{\F}$. Then the pair $(\widehat{\F},\hat{m})$ is called a \textit{vector measure algebra} induced by $m$. Denote by $\F_E=\{ A\cap E\mid A\in \F \}$, a $\sigma$-\hspace{0pt}algebra of $E\in \F$ inherited from $\F$, and $(\widehat{\F_E},\hat{m})$ a vector measure algebra induced by the restriction of $m$ to $\F_E$. Then $\widehat{\F_E}$ is the principal ideal of $\widehat{\F}$ generated by the element $\widehat{E}\in \widehat{\F}$. A $\sigma$-\hspace{0pt}algebra $\F$ is \textit{$m$-\hspace{0pt}essentially countably generated} if there exists a countably generated sub $\sigma$-\hspace{0pt}algebra $\G$ of $\F$ such that $\widehat{\G}=\widehat{\F}$. 

A set $A\in \F$ is an \textit{atom} of $m$ if $m(A)\ne \bold{0}$ and $\widehat{A}$ is an atom of $\widehat{\F}$. If $m$ has no atom, it is said to be \textit{nonatomic}. The Maharam type of a vector measure $m:\F\to X$ is defined to be $\kappa(\widehat{\F})$. Thus, $\kappa(\widehat{\F})$ is countable if and only if $\F$ is $m$-\hspace{0pt}essentially countably generated. Hence, $m$ is nonatomic if and only if $\kappa(\widehat{\F_E})$ is infinite for every $m$-\hspace{0pt}nonnull $E\in \F$. A vector measure space $(\Omega,\F,m)$ (or\ a vector measure $m$) is \textit{saturated} if $\kappa(\widehat{\F_E})$ is uncountable for every $m$-\hspace{0pt}nonnull $E\in \F$. Let $Y$ be a lcHs. A vector measure $n:\F\to Y$ is \textit{absolutely continuous} (or \textit{$m$-\hspace{0pt}continuous}) with respect to a vector measure $m:\F\to X$ if every $m$-\hspace{0pt}null set is $n$-\hspace{0pt}null. It is evident that for every lcHs $Y$ a vector measure $n:\F\to Y$ is saturated (resp.\ nonatomic) if and only if there exists a  saturated (resp.\ nonatomic) vector measure $m:\F\to X$ with respect to which $n$ is absolutely continuous.

\section{The $L^1$-Space of Vector Measures}
\subsection{Integrals with Respect to Vector Measures}
The following Pettis-\hspace{0pt}like notion of the integral of measurable functions with respect to a vector measure was introduced in \citet{bo59} and elaborated independently by \citet{kl70,le70,le72}. For a detailed treatment of this integral, see \cite{kk75,osr08,pa08}. 

\begin{df}
Let $\langle x^*,m \rangle:\F \to \R$ be the scalar measure defined by $\langle x^*,m \rangle(A):=\langle x^*,m(A) \rangle$ with $x^*\in X^*$ and $A\in \F$. A measurable function $f:\Omega\to \R$ is \textit{$m$-\hspace{0pt}integrable} if it is integrable with respect to the scalar measure $\langle x^*,m \rangle$ for every $x^*\in X^*$ and for every $A\in \F$ there exists a vector $x_A\in X$ such that
$$
\left\langle x^*,x_A \right\rangle=\int_Afd\langle x^*,m \rangle \quad \text{for every $x^*\in X^*$}. 
$$ 
\end{df}

Since the dual space $X^*$ is a total family of linear functionals on a lcHs $X$, the vector $x_A$ is unique, which we denote by $\int_Afdm$. Unlike the definition of the integrals with respect to $m$ formulated in \cite{ks14} (see also the references therein), no assumption is made about the completeness of $X$ and the existence of a scalar measure with respect to which $m$ is absolutely continuous is unnecessary. 

The vector space of all $m$-\hspace{0pt}integrable functions is denoted by $L(m)$. Let $\P$ be a separating family of seminorms in $X$ that generates the locally convex topology $\tau$. Let $U_p^\circ$ denotes the polar of the set $U_p=\{ x\in X\mid p(x)\le 1 \}$, that is, $U_p^\circ=\{ x^*\in X^*\mid |\langle x^*,x \rangle|\le 1 \ \forall x\in U_p \}$. Each $p\in \P$ induces a seminorm $p(m)$ in $L(m)$ via the formula
$$
p(m)(f)=\sup_{x^*\in U_p^\circ}\int |f|d|\langle x^*,m \rangle|, \quad f\in L(m),
$$
where $|\langle x^*,m \rangle|$ denotes the variation of the scalar measure $\langle x^*,m \rangle$. The above seminorms turn $L(m)$ into a locally convex space. The quotient space of $L(m)$ modulo the subspace $\bigcap_{p\in \P}p(m)^{-1}(0)$ of all $m$-\hspace{0pt}null functions is denoted by $L^1(m)$, which is a lcHs with its topology denoted by $\tau(m)$. 

Denote by $L^\infty(m)$ (the equivalence classes of) the space of all $m$-\hspace{0pt}essentially bounded measurable functions $f$ on $\Omega$, endowed with the $m$-\hspace{0pt}essentially supremum norm
$$
\| f \|_\infty=\inf\{ \alpha>0\mid \text{$\{ \omega\in \Omega\mid |f(\omega)|>\alpha \}$ is $m$-\hspace{0pt}null} \}.
$$
Recall that $X$ is said to be \textit{sequentially complete} if every Cauchy sequence in $X$ converges. If $X$ is a sequentially complete lcHs, then $L^\infty(m)\subset L^1(m)$ (see \cite[Lemma II.3.1]{kk75} or \cite[Theorem 4.1.9$'$]{pa08}) and $(L^\infty(m),\| \cdot \|_\infty)$ is a Banach space (see \cite[Theorem 4.5.8]{pa08}). 

A linear operator $T_m:L^1(m)\to X$ defined by $T_mf=\int fdm$ for $f\in L^1(m)$ is called an \textit{integration operator} of $m$. We also denote $T_mf$ by $m(f)$. Since $L^\infty(m)\subset L^1(m)$ whenever $X$ is sequentially complete, one can restrict the integration operator $T_m$ to $L^\infty(m)$ endowed with the $m$-\hspace{0pt}essential sup norm. Hence, if $X$ is a sequentially complete lcHs, then the integration operator $T_m:L^\infty(m)\to X$ is continuous (see \cite[Lemma II.3.1]{kk75}). Moreover, the following continuity result of the integration operator is true without any completeness assumption on $X$. 

\begin{lem}
\label{lem1}
The integration operator $T_m:L^1(m)\to X$ is continuous for the weak topologies\footnote{See \cite{ok93} for the specification of the dual space of $L^1(m)$.}of $L^1(m)$ and $X$.
\end{lem}

\begin{proof}
We first show that the finite signed measure $x^*m$ is a continuous linear functional on $L^1(m)$ for every $x^*\in X^*$ with respect to $\tau(m)$-\hspace{0pt}topology. To this end, let $\{ f_\alpha \}$ be a net in $L^1(m)$ such that $p(m)(f_\alpha-f)$ for every $p\in \P$. If $x^*\in X^*$ vanishes on $U_p$, then $x^*\in U_p^\circ$, and hence, we obtain
$$
\left| \int f_\alpha d\langle x^*,m \rangle-\int fd\langle x^*,m \rangle \right|\le \int|f_\alpha-f|d|\langle x^*,m \rangle|\le p(m)(f_\alpha-f)\to 0.
$$
If $x^*\ne 0$ on $U_p$, define $\tilde{p}(m)(x^*)=\sup_{x\in U_p}|\langle x^*,x \rangle|$. By normalization, we have $y^*:=x^*/\tilde{p}(m)(x^*)\in U_p^\circ$, and hence
\begin{align*}
\left| \int f_\alpha d\langle x^*,m \rangle-\int fd\langle x^*,m \rangle \right|
& \le \tilde{p}(m)(x^*)\int|f_\alpha-f|d|\langle y^*,m \rangle| \\
& \le [\tilde{p}(m)(x^*)][p(m)(f_\alpha-f)]\to 0.
\end{align*}
Therefore, $\int f_\alpha d\langle x^*,m \rangle\to \int f d\langle x^*,m \rangle$ for every $x^*\in X^*$. This means that $\langle x^*,m \rangle$ is an element of the dual space $(L^1(m))^*$ for every $x^*\in X^*$. Let $\{ g_\alpha \}$ be a net in $L^1(m)$ that converges weakly to $g\in L^1(m)$. Then for every $x^*\in X^*$, we have
\begin{align*}
\langle x^*,T_mg_\alpha \rangle=\int g_\alpha d\langle x^*,m \rangle\to \int g d\langle x^*,m \rangle=\langle x^*, T_mg \rangle.
\end{align*}
Hence, $T_m$ is continuous for the weak topologies of $L^1(m)$ and $X$.
\end{proof}

\subsection{Completeness of $L^1(m)$}
A measure space $(\Omega,\F,\mu)$ (possibly $\mu$ is an infinite measure) is \textit{localizable} if for every continuous linear functional $\varphi$ on $L^1(\mu)$ there is a bounded measurable function $g:\Omega\to \R$ such that $\varphi(f)=\int fgd\mu$ for every $f\in L^1(\mu)$ (see \cite[Section I.3]{kk75}). When $\mu$ is $\sigma$-\hspace{0pt}finite, localizability is automatically satisfied because the dual of $L^1(\mu)$ is $L^\infty(\mu)$. It is well known that such duality is no longer true if $\mu$ is not $\sigma$-\hspace{0pt}finite. A measure space $(\Omega,\F,\mu)$ is localizable if and only if its measure algebra is a complete Boolean algebra as a partially ordered set (see \cite[Theorem 5.1]{se50}). 

For $p\in \P$, the \textit{$p$-\hspace{0pt}semivariation} of a vector measure $m:\F\to X$ is a set function $\| m \|_p:\F\to \R$ defined by
$$
\| m \|_p(A)=\sup_{x^*\in U_p^\circ}|\langle x^*,m \rangle|(A), \quad A\in \F. 
$$
The $p$-\hspace{0pt}semivariation $\| m \|_p$ of $m$ is bounded, monotone and countably subadditive with the following estimate (see \cite[Lemma II.1.2]{kk75}):
$$
\sup_{E\subset A}p(m(E))\le\| m \|_p(A)\le 2\sup_{E\subset A}p(m(E)). 
$$
A finite (scalar) measure $\mu_p$ is a \textit{$p$-\hspace{0pt}control measure} of $m$ whenever $\mu_p(A)=0$ if and only if $\| m \|_p(A)=0$. If $X$ is a lcHs, then for every $p\in \P$ there exists a $p$-\hspace{0pt}control measure of $m$ (see \cite[Corollaries II.1.2 on p.\,19 and II.1.1 on p.\,21]{kk75}). A vector measure $m:\F\to X$ is \textit{absolutely continuous} with respect to a scalar measure $\mu$ if $\mu(A)=0$ implies that $m(A)=\bold{0}$. A finite measure $\mu$ is a \textit{control measure} of $m$ whenever $\mu(A)=0$ if and only if $m(A\cap E)=\bold{0}$ for every $E\in \F$. 

Since each characteristic function of $\F$ is identified with an element of $L^1(m)$,  we can restrict the locally convex Hausdorff topology $\tau(m)$ of $L^1(m)$ to the vector measure algebra $(\widehat{\F},\hat{m})$ of $m$. Thus, the relative $\tau(m)$-\hspace{0pt}topology on $\widehat{\F}$ is generated by a family of semimetrics $\{ d_p\mid p\in \P \}$ on $\widehat{\F}$ by the formula
$$
d_p(\widehat{A},\widehat{B}):=p(m)(\chi_A-\chi_B)=\| m \|_p(A\triangle B), \quad A,B\in \F.
$$
Denote by $L^1_E(m)=\{ f\chi_E\mid f\in L^1(m) \}$ the vector subspace consisting of $m$-\hspace{0pt}integrable functions on $\Omega$ restricted to $E\in \F$ and similarly, $L^\infty_E(m)=\{ f\chi_E\mid f\in L^\infty(m) \}$ the vector subspace consisting of $m$-\hspace{0pt}essentially bounded functions on $\Omega$ restricted to $E$. 

Without completeness, the function space $L^1(m)$ would be useless in the limit operation of integrals. To overcome this difficulty, \citet{kk75} introduced the notion of closed vector measures in lcHs.  

\begin{df}
A vector measure $m:\F\to X$ is \textit{closed} if $\widehat{\F}$ is complete in the $\tau(m)$-\hspace{0pt}topology. 
\end{df}

\begin{quote}
Closed vector measures are those for which most of the classical theory of $L^1$ spaces carries over, especially results concerning completeness. The phenomenon of non-\hspace{0pt}closed measures is observable only if the range space is not metrizable. (\cite[p.\,67]{kk75}.)
\end{quote}

A sufficient condition for the closedness of $m$ given by \citet[Proposition 1]{ri84} is the metrizability of the range $m(\F)$ in a lcHs $X$. In particular, if $X$ is a Fr\'echet space, then $L^1(m)$ is also a Fr\'echet space (see \cite[Theorems IV.4.1 and IV.7.1]{kk75} and \cite[Theorem 1]{fnr}). The significance of the notion of closed vector measures is exemplified by the next characterization of the completeness of $L^1(m)$ attributed to \cite{kk75}, which can be slightly generalized as the current form, where the assumption of quasicompleteness is replaced by that of sequentially completeness (see \cite[Proposition 1]{ri83} and \cite[Proposition 3]{fnr}).  

\begin{pro}
\label{pro1}
Let $X$ be a sequentially complete lcHs. Then a vector measure $m:\F\to X$ is closed if and only if $L^1(m)$ is $\tau(m)$-\hspace{0pt}complete.
\end{pro}
\noindent
For other characterizations of closed vector measures, see \cite{or95,or99,ri84,ri90}. 

It follows immediately from Proposition \ref{pro1} that a vector measure is closed if it has a control measure. A lcHs $X$ has the \textit{Bartle--\hspace{0pt}Dunford--\hspace{0pt}Schwartz (BDS) property} if every $X$-\hspace{0pt}valued vector measure has a control measure. As demonstrated in \citet[Example 2.1]{ks14}, a lcHs $X$ has the BDS property for each of the following cases: (i) $X$ is metrizable; (ii) $X$ is Suslin; (iii) $X^*$ is weakly$^*$ separable. 

\begin{ex}
A vector measure $m:\F\to X$ is \textit{countably determined} if there exists a sequence $\{ x_n^* \}$ in $X^*$ such that a set $A\in \F$ is $m$-\hspace{0pt}null if it is $\langle x_n^*,m \rangle$-\hspace{0pt}null for each $n=1,2,\dots$. A countably determined vector measure $m$ in a lcHs $X$ has a control measure. Indeed, its control measure is given by
$$
\mu(A)=\sum_{n=1}^\infty \frac{|\langle x_n^*,m \rangle|(A)}{2^n(1+|\langle x_n^*,m \rangle|(\Omega))},\quad A\in \F.
$$
\end{ex}

\begin{rem}
At a first glance, the closedness of vector measures seems innocuous, but it demands a lot because a vector measure $m:\F\to X$ is closed if and only if there exists a localizable measure $\mu$ such that $m$ is $\mu$-\hspace{0pt}continuous (see \cite[Theorem IV.7.3]{kk75} and \cite[Lemma 11 and Corollary 13]{kl77}). Indeed, it is ``nothing" but absolute continuity in disguise! Without $\sigma$-\hspace{0pt}finiteness, the localizable measures nevertheless would not play a significant role as in the control measures of a vector measure because of the lack of the duality between  $L^1(\mu)$ and $L^\infty(\mu)$. 
\end{rem}

\subsection{Separability of $L^1(m)$}
The following observation due to \citet[Proposition 1A]{ri92} (see also \cite[Theorems 4.6.2 and 4.6.3]{pa08}) plays a crucial role to develop the notion of saturation for vector measures in lcHs. 

\begin{pro}
\label{pro2}
If $X$ is a sequentially complete lcHs, then $L^1(m)$ is separable if and only if $\widehat{\F}$ is separable. 
\end{pro}

Indeed, for a scalar measure $\mu$, the non-\hspace{0pt}separability of $L^1(\mu)$ is a defining property for the saturation of the measure space $(\Omega,\F,\mu)$ (see \cite[331O and 365X(p)]{fr12}, \cite[Corollary 4.5]{hk84}, and \cite[Fact 2.5]{ks09}). 

\begin{thm}
\label{thm2}
Let $X$ be a sequentially complete lcHs and $m:\F\to X$ be a vector measure. Then the following conditions are equivalent:
\begin{enumerate}[\rm(i)] 
\item $(\Omega,\F,m)$ is saturated;
\item $\widehat{\F_E}$ is non-\hspace{0pt}separable for every $m$-\hspace{0pt}nonnull $E\in \F$; 
\item $L^1_E(m)$ is non-\hspace{0pt}separable for every $m$-\hspace{0pt}nonnull $E\in \F$.
\end{enumerate} 
\end{thm}

\begin{proof}
(ii) $\Leftrightarrow$ (iii): See Proposition \ref{pro2}. 

(iii) $\Rightarrow$ (i): If the Maharam type of $\widehat{\F_E}$ is countable for some $m$-\hspace{0pt}nonnull $E\in \F$, then there is a countable subset $\widehat{\U}$ of $\widehat{\F_E}$ that completely generates $\widehat{\F_E}$. By virtue of the axiom of choice, there is a choice function $\varphi: \widehat{\F_E}\to \F_E$ such that $\varphi(\widehat{A})\in \widehat{A}$ for every $\widehat{A}\in \widehat{\F_E}$. Let $\G$ be the subalgebra of $\F_E$ completely generated by $\varphi(\widehat{\U})$. By construction, $\G$ is $m$-\hspace{0pt}essentially countably generated satisfying $\widehat{\G}=\widehat{\F_E}$, and hence, $L^1_E(m)$ is separable (see \cite[Proposition 2]{ri92}). 

(i) $\Rightarrow$ (iii): If the Maharam type of $\widehat{\F_E}$ is uncountable for every $m$-\hspace{0pt}nonnull $E\in \F$, then the cardinality of $\widehat{\F_E}$ is uncountable. Suppose, to the contrary, that $L^1_E(m)$ is separable for some $m$-\hspace{0pt}nonnull $E\in \F$. By Proposition \ref{pro2}, $\widehat{\F_E}$ is separable, so there exists a countable subset $\G$ of $\F_E$ such that for every $A\in \F_E$, every $\varepsilon_1,\dots,\varepsilon_k>0$ and every $p_1,\dots,p_k\in \P$ there exists $B\in \G$ satisfying $\| m \|_{p_i}(A\triangle B)<\varepsilon_i$ for each $i=1,\dots,k$. On the other hand, since $\widehat{\F_E}\setminus \widehat{\G}$ is uncountable, there exists an $m$-\hspace{0pt}nonnull set $A\in \F_E\setminus \G$ with $\widehat{A}\in \widehat{\F_E}\setminus \widehat{\G}$. Since $\P$ is a separating family of seminorms, we can take a seminorm $p\in \P$ with $p(m(A))>0$. Take any $B\in \G$. Then $A\cap B=\emptyset$, and hence, $\| m \|_p(A\triangle B)=\| m \|_p(A\cup B)\ge \| m \|_p(A)\ge p(m(A))$ for every $B\in \G$, a contradiction. Therefore, $L^1_E(m)$ is non-\hspace{0pt}separable. 
\end{proof}

\section{Lyapunov Convexity Theorem in LcHs}
\subsection{Lyapunov Measures and Lyapunov Operators}
Following \cite{ks13,ks15}, we characterize the Lyapunov convexity theorem in terms of the integration operator. 

\begin{df}
A vector measure $m:\F\to X$ is a \textit{Lyapunov measure} if for every $E\in \F$ the set $m(\F_E)$ is weakly compact and convex in $X$.
\end{df}

\begin{df}
The integration operator $T_m:L^\infty(m)\to X$ is said to be:
\begin{enumerate}[(i)]
\item a \textit{nonatomic operator} if for every $m$-\hspace{0pt}nonnull $E\in \F$ and every neighborhood $U$ of the origin in $X$ there exists $f\in L^\infty_E(m)\setminus \{ 0 \}$ with signed values $\{ -1,0,1 \}$ such that $T_mf\in U$;
\item a \textit{Lyapunov operator} of $m$ if for every $m$-\hspace{0pt}nonnull $E\in \F$ the restriction $T_m:L^\infty_E(m)\to X$ is not injective.
\end{enumerate}
\end{df}

\begin{thm}
\label{thm3}
Let $X$ be a sequentially complete lcHs and $m:\F\to X$ be a vector measure. Then $m$ is nonatomic if and only if $T_m:L^\infty(m)\to X$ is a nonatomic operator.
\end{thm}

\begin{proof}
Suppose that $T_m$ is a nonatomic operator. If $m$ has an atom $E\in \F$, then $m(E)\ne \bold{0}$. Thus, for every neighborhood $U$ of $\bold{0}$ there exists $f\in L^\infty_E(m)\setminus \{ 0 \}$ with signed values $\{ -1,0,1 \}$ such that $T_mf\in U\cap (-U)$. Since measurable functions are constant on atoms of $m$, either $f=\chi_E$ or $f=-\chi_E$. We thus obtain $m(E)\in \{ \pm T_mf \}\subset U\cap(-U)$ for every neighborhood $U$ of $\bold{0}$, and hence, $m(E)=\bold{0}$, a contradiction.

Conversely, suppose that $T_m$ is not a nonatomic operator. Then there exists $E\in \F$ with $m(E)\ne \bold{0}$ and a convex, balanced, absorbing neighborhood $U$ of $\bold{0}$ such that $T_mf\not\in U$ for every $f\in L^\infty_E(m)\setminus \{ 0 \}$ with signed values $\{ -1,0,1 \}$. (Such a neighborhood can be taken as $U\cap (-U)$ as above). Since there exists $p\in \P$ such that $U=U_p$ (see \cite[Theorems 1.34 and 1.35]{ru73}), for every $m$-\hspace{0pt}nonnull $A\in \F_E$, we have $m(A)=T_m\chi_A\not\in U_p$. Thus, $\| m \|_p(A)\ge p(m(A))>1$ for every $m$-\hspace{0pt}nonnull $A\in \F_E$. If $E$ is not an atom of $m$, then there exists $A\in \F_E$ such that $m(A)\ne m(E)$ and $m(A)\ne \bold{0}$. Take any $p$-\hspace{0pt}control measure $\mu_p$ of $m$. By the $\mu_p$-\hspace{0pt}continuity of $\| m \|_p$, we have $\mu_p(A)>0$. Hence, there exists $\delta>0$ such that for every $B\in \F$ with $\mu_p(A\cap B)<\delta$, we have $\| m \|_p(A\cap B)<1$, a contradiction. Therefore, $E$ is an atom of $m$.
\end{proof}

The range of a Lyapunov measure $m$ is weakly compact and convex in $X$ in view of $m(\F)=m(\F_\Omega)$. If $m$ has an atom $E\in \F$, then evidently, $m(\F_E)$ is not convex in $X$. Therefore, every Lyapunov measure is nonatomic. As the next result demonstrates, the nonatomicity of vector measures is reinforced as well by the notion of Lyapunov operators (see \cite[Theorem 3.2]{ks13}). 

\begin{thm}
\label{thm4}
Let $X$ be a sequentially complete lcHs and $m:\F\to X$ be a vector measure. If $T_m:L^\infty(m)\to X$ is a Lyapunov operator, then it is a nonatomic operator.
\end{thm}

The following result is due to \cite[Theorem V.1.1]{kk75}. 

\begin{pro}
\label{pro3}
Let $X$ be a quasicomplete lcHs and $m:\F\to X$ be a closed vector measure. Then $m$ is a Lyapunov measure if and only if $T_m:L^\infty(m)\to X$ is a Lyapunov operator. The range of a Lyapunov measure $m$ is given by
$$
m(\F)=\{ m(f) \in X\mid 0\le f\le 1,\,f\in L^\infty(m) \}.
$$
\end{pro}

\subsection{Saturation: A Sufficiency Theorem}
\begin{lem}
\label{lem2}
Let $X$ be a sequentially complete lcHs and $m:\F\to X$ be a closed vector measure. Then a bounded closed convex subset of $L^\infty(m)$ is weakly compact in $L^1(m)$. 
\end{lem}

\begin{proof}
Let $K$ be a bounded, closed, convex subset of $L^\infty(m)$. By Proposition \ref{pro1}, $K$ is $\tau(m)$-\hspace{0pt}complete in $L^1(m)$. Since the $\tau(m)$-\hspace{0pt}topology of $L^1(m)$ is generated by the family of seminorms $\{q_{p,x^*}\mid p\in \P,\,x^*\in U_p^\circ \}$ defined by $q_{p,x^*}(f)=\int|f|d|\langle x^*,m \rangle|$, the quotient space $L^1(m)/q_{p,x^*}^{-1}(0)$ is a vector subspace of  $L^1(|\langle x^*,m \rangle|)$ endowed with the $L^1(|\langle x^*,m \rangle|)$-\hspace{0pt}norm. Let $\pi_{p,x^*}:L^1(m)\to L^1(m)/q_{p,x^*}^{-1}(0)$ be the natural projection. Then $\pi_{p,x^*}(K)$ is a closed subset of $L^1(|\langle x^*,m \rangle|)$. Since the boundedness of $K$ in $L^\infty(m)$ implies that $\pi_{p,x^*}(K)$ is uniformly integrable in the sense that
$$
\lim_{|\langle x^*,m \rangle|(A)\to 0}\sup_{f\in \pi_{p,x^*}(K)}\int_A|f|d|\langle x^*,m \rangle|=0
$$
by the Dunford--\hspace{0pt}Pettis criterion (see \cite[Corollary IV.8.11]{ds58}), $\pi_{p,x^*}(K)$ is weakly sequentially compact, and hence, weakly compact in $L^1(|\langle x^*,m \rangle|)$ for every $x^*\in U_p^\circ$ and $p\in \P$ in view of the Eberlein--\hspace{0pt}\u{S}mulian theorem (see \cite[Theorem V.6.1]{ds58}). By \cite[Theorem I.1.1]{kk75}, the weak compactness of $\pi_{p,x^*}(K)$ in $L^1(|\langle x^*,m \rangle|)$ for every $p\in \P$ and $x^*\in X^*$ implies the weak compactness of $K$ in $L^1(m)$. 
\end{proof}

The \textit{density} of a topological space $S$, denoted by $\mathrm{dens}\,S$, is the smallest cardinal of any dense subset of $S$. The density of a lcHs $X$ is equal to the \textit{topological dimension} of $X$, i.e., the smallest cardinal of any set whose linear span is dense in $X$ whenever $\mathrm{dens}\,X$ is infinite. 

\begin{lem}
\label{lem7}
$\kappa(\widehat{\F_E})\le \mathrm{dens}\,\widehat{\F_E}$ for every $m$-\hspace{0pt}nonnull $E\in \F$. 
\end{lem}

\begin{proof}
Let $E\in \F$ be an $m$-\hspace{0pt}nonnull set and $\{ \widehat{A}_\alpha \}_{\alpha<\mathrm{dens}\,\widehat{\F_E}}$ be a family of elements in $\widehat{\F_E}$ such that its $\tau(m)$-\hspace{0pt}closure coincides with $\widehat{\F_E}$. Denote by $\widehat{\U}$ the subalgebra of $\widehat{\F_E}$ completely generated by $\{ \widehat{A}_\alpha \}_{\alpha<\mathrm{dens}\,\widehat{\F_E}}$. Since $\{ \widehat{A}_\alpha \}_{\alpha<\mathrm{dens}\,\widehat{\F_E}}$ is contained in $\widehat{\U}$, the $\tau(m)$-\hspace{0pt}closure $\mathrm{cl}\,\widehat{\U}$ of $\widehat{\U}$ coincides with $\widehat{\F_E}$. If we demonstrate that $\widehat{\U}=\widehat{\F_E}$, then we have $\kappa(\widehat{\F_E})\le \mathrm{dens}\,\widehat{\F_E}$. To this end, it suffices to show the $\tau(m)$-\hspace{0pt}closedness of $\widehat{\U}$. Let $\{ \widehat{A}^\nu \}$ be a net in $\widehat{\U}$ converging to $\widehat{A}$. Since 
$\mathrm{cl}\,\widehat{\U}=\widehat{\F_E}$, there exists $\widehat{B}^\nu\in \widehat{\U}$ such that $\widehat{B}^\nu\le \widehat{A}\wedge \widehat{A}^\nu\in \widehat{\F_E}$ for each $\nu$ and the net $\{ \widehat{B}^\nu \}$ converges to $\widehat{A}$. Extracting a subnet from $\{ \widehat{B}^\nu \}$ (which we do not relabel), one may assume that $\{ \widehat{B}^\nu \}$ is upward directed in $\widehat{\U}$ with $\lim_\nu\widehat{B}^\nu=\sup_\nu \widehat{B}^\nu=\widehat{A}$. Since $\widehat{\U}$ is order closed, we have $\widehat{A}\in \widehat{\U}$. 
\end{proof}

\begin{lem}
\label{lem3}
If $X$ is a sequentially complete lcHs and $m:\F\to X$ is a closed vector measure such that $\mathrm{dens}\,X<\kappa(\widehat{\F_E})$ for every $m$-\hspace{0pt}nonnull $E\in \F$, then $T_m:L^\infty(m)\to X$ is a Lyapunov operator.
\end{lem}

\begin{proof}
Suppose to the contrary that $T_m$ is not a Lyapunov operator. Then there exists an $m$-\hspace{0pt}nonnull set $E\in \F$ such that the restriction $T_m:L^\infty_E(m)\to X$ is an injection. Let $\B$ denote the closed unit ball in $L^\infty(m)$ and set $\B_E=\B\cap L^\infty_E(m)$. Since $L^\infty_E(m)$ is a closed vector subspace of $L^1(m)$, by Lemma \ref{lem2}, $\B_E$ is a weakly compact subset of $L^1(m)$. Since the integration operator $T_m:L^1(m)\to X$ is continuous for the weak topologies of $L^1(m)$ and $X$ by Lemma \ref{lem1}, the restriction $T_m$ to $\B_E$ is a homeomorphism between $\B_E$ and $T_m(\B_E)$. It follows from the convexity and weak closedness of $T_m(\B_E)$ that $T_m(\B_E)$ is also closed in the strong topology of $X$. Let $\{ x_\alpha \}_{\alpha<\mathrm{dens}\,X}$ be a dense subset of $T_m(\B_E)$ and define $f_\alpha=T^{-1}_mx_\alpha$. Then $\{ f_\alpha \}_{\alpha<\mathrm{dens}\,X}$ is a weakly dense subset of $\B_E$ in $L^1(m)$. Since $\B_E$ is convex, it is $\tau(m)$-\hspace{0pt}closed, and hence, $\{ f_\alpha \}_{\alpha<\mathrm{dens}\,X}$ is also a $\tau(m)$-\hspace{0pt}dense subset of $\B_E$. Since  $\widehat{\F_E}\subset \B_E$, by Lemma \ref{lem7}, we have $\kappa(\widehat{\F_E})\le \mathrm{dens}\,\widehat{\F_E}\le \mathrm{dens}\,\B_E\le \mathrm{dens}\,X$, a contradiction to the hypothesis. 
\end{proof}

An immediate consequence of Proposition \ref{pro3} and Lemma \ref{lem3} is the following version of the Lyapunov theorem, which is a further generalization of \cite{gp13,ks13,ks15}. 

\begin{thm}
\label{thm5}
If $X$ is a quasicomplete lcHs and $m:\F\to X$ is a closed vector measure such that $\mathrm{dens}\,X<\kappa(\widehat{\F_E})$ for every $m$-\hspace{0pt}nonnull $E\in \F$, then $m$ is a Lyapunov measure with its range given by
$$
m(\F)=\{ m(f)\in X\mid 0\le f\le 1,\,f\in L^\infty(m) \}.
$$
\end{thm}
\noindent
In particular, if $X$ is separable and $m$ is saturated, then the density hypothesis of Theorem \ref{thm5} is automatic because $\mathrm{dens}\,X=\aleph_0<\aleph_1\le \kappa(\widehat{\F_E})$ for every $m$-\hspace{0pt}nonnull $E\in \F$.  

\begin{rem}
The \textit{algebraic dimension} of a lcHs $X$ is the cardinality of a Hamel basis in $X$. We denote by $\dim X$ the algebraic dimension of $X$. Consider the dimensionality condition:
$$
\dim L^\infty_E(m)>\dim X \quad \text{for every $m$-\hspace{0pt}nonnull $E\in \F$}.
$$
This is obviously a sufficient condition for $T_m$ to be a Lyapunov operator. As evident from Lemma \ref{lem3}, the dimensionality condition is sharpened to the density condition:
$$
\mathrm{dens}\,L^\infty_E(m)>\mathrm{dens}\,X \quad \text{for every $m$-\hspace{0pt}nonnull $E\in \F$}.
$$
For the discussion of the dimensionality and density conditions, see \cite[Remark 3.2]{ks13}. 
\end{rem}

\subsection{Saturation: A Necessity Theorem}
Let $Y$ be a lcHs. Denote by $\mathit{ca}(\F,m,Y)$ the space of $Y$-\hspace{0pt}valued closed vector measures on $\F$ which are absolutely continuous with respect to a vector measure $m:\F\to X$. 

\begin{thm}
\label{thm8}
Let $X$ be a quasicomplete separable lcHs, $Y$ be a quasicomplete, separable lcHs with a vector subspace that is isomorphic to an infinite-\hspace{0pt}dimensional Banach space, and $m:\F\to X$ be a closed vector measure. Then the following conditions are equivalent:
\begin{enumerate}[\rm (i)]
\item $(\Omega,\F,m)$ is saturated;
\item $T_n:L^\infty(n)\to Y$ is a Lyapunov operator for every $n\in \mathit{ca}(\F,m,Y)$;
\item Every vector measure in $\mathit{ca}(\F,m,Y)$ is saturated;
\item Every vector measure in $\mathit{ca}(\F,m,Y)$ is a Lyapunov measure.
\end{enumerate}
\end{thm}

\begin{proof}
(i) $\Rightarrow$ (iii): Trivial.

(iii) $\Rightarrow$ (iv): See Lemma \ref{lem3} and Theorem \ref{thm5}.

(iv) $\Rightarrow$ (i): If $m$ is not saturated, then there exists an $m$-\hspace{0pt}nonnull set $E\in \F$ such that $\tau(\widehat{\F_E})$ is a countable cardinal, where $(\widehat{\F_E},\hat{m})$ is the restriction of the vector measure algebra $(\widehat{\F},\hat{m})$ to $E$. This means that $\F_E$ is $m$-\hspace{0pt}essentially countably generated, and hence, there exists a countably generated sub $\sigma$-\hspace{0pt}algebra $\G$ of $\F_E$ such that $\widehat{\G}=\widehat{\F_E}$. Since the restriction of $m$ to $\F_E$ is an $X$-\hspace{0pt}valued Lyapunov measure on $\F_E$, it is nonatomic on $\F_E$ by Theorems \ref{thm3} and \ref{thm4}. Take any $p$-\hspace{0pt}control measure $\mu_p$ of $m$. It follows from $\widehat{\G}=\widehat{\F_E}$ that $m$ is nonatomic on $\G$, and hence, $\mu_p$ is also nonatomic on $\G$ because of the absolutely continuity of $\mu_p$ with respect to $m$. Therefore, $(E,\G,\mu_p)$ is a countably generated, nonatomic, finite measure space, which is not obviously saturated. Let $\tilde{Y}$ be a vector subspace of $Y$ that is isomorphic to an infinite-\hspace{0pt}dimensional Banach space. By \cite[Lemma 4.1]{ks13}, there exists a non-\hspace{0pt}Lyapunov measure $n_F\in \mathit{ca}(\G_F,\mu_p, \tilde{Y})$ for some $F\subset E$ with $\mu_p(F)>0$. Extend $n_F$ from $\G_F$ to $\F$ by $n(A)=n_F(A\cap F)$ for $A\in \F$. Indeed, $n$ is closed since it absolutely continuous with respect to $\mu_p$. Hence, $n\in \mathit{ca}(\F,m,Y)$ is not a Lyapunov measure.

(i) $\Leftrightarrow$ (ii): See Proposition \ref{pro3}.
\end{proof}

By taking $m$ as a scalar measure, we can recover obviously the necessity result explored in \cite{ks13,ks15}.

\section{Further Characterization of Saturation}
\subsection{The Bang-Bang Principle in LcHs}
Let $I$ be an arbitrary subset of the set $\mathbb{N}$ of natural numbers and $X_i$ be a lcHs for each $i\in I$. Denote by $\prod_{i\in I}X_i$ the product space of $X_i$ consisting of all mappings $I\ni i\mapsto x_i\in X_i$, endowed with the product topology (given by the pointwise convergence in $X_i$ for each $i\in I$). Let $m:\F\to \prod_{i\in I}X_i$ be a vector measure with a component measure $m_i:\F\to X_i$ for $i\in I$. Given a vector measure space $(\Omega,\F,m)$, two measurable functions $f:\Omega\to \R^I$ and $g:\Omega\to \R^I$ are regarded as equivalent if $f(\omega)=g(\omega)$ except on the $m$-\hspace{0pt}null set. 

The infinite-\hspace{0pt}dimensional control systems under scrutiny here are described by the lcHs $\prod_{i\in I}X_i$, the vector measure space $(\Omega,\F,m)$, and a multifunction $\Gamma:\Omega\twoheadrightarrow \R^I$. Denote by $\mathcal{S}^1_\Gamma$ the set of measurable selectors $f:\Omega\to \R^I$ from $\Gamma$ such that each component function $f_i:\Omega\to \R$ is $m_i$-\hspace{0pt}integrable. Thus, $\mathcal{S}^1_\Gamma$ is regarded as a subset of the product space $\prod_{i\in I}L^1(m_i)$. For each $f_i\in L^1(m_i)$, let $m_i(f_i)=\int f_i dm_i\in X_i$. Then the integral of $\Gamma$ with respect to $m$ is given by
$$
\int \Gamma dm=\{ (m_i(f_i))_{i\in I}\in X \mid f=(f_i)_{i\in I}\in \mathcal{S}_\Gamma^1 \}.
$$

\begin{df}
A vector measure $m:\F\to \prod_{i\in I}X_i$ satisfies the \textit{bang-\hspace{0pt}bang principle} for a multifunction $\Gamma:\Omega\twoheadrightarrow \R^I$ if $\int \Gamma dm=\int \mathrm{ex}\,\Gamma dm$, where $\mathrm{ex}\,\Gamma:\Omega\twoheadrightarrow \R^I$ is the multifunction given by the extreme points of $\Gamma(\omega)$ at each $\omega\in \Omega$. 
\end{df}

A multifunction $\Gamma:\Omega\twoheadrightarrow \R^I$ is \textit{bounded} if there exists a bounded set $K\subset \R^I$ such that $\Gamma(\omega)\subset K$ for every $\omega\in \Omega$. Denote by $\M_{\mathit{bfc}}(\Omega,\R^I)$ the set of graph measurable, bounded multifunctions with closed convex values from $\Omega$ to $\R^I$. 

The next result is a special case of \cite[Theorems IV.14 and IV.15]{cv77}. 

\begin{lem}
\label{lem5}
For every $\Gamma\in \M_{\mathit{bfc}}(\Omega,\R^I)$, the following conditions hold.
\begin{enumerate}[\rm(i)]
\item $\mathrm{ex}\,\mathcal{S}_{\Gamma}^1=\mathcal{S}_{\mathrm{ex}\,\Gamma}^1$. 
\item For every $f\in \mathcal{S}_{\Gamma}^1$ there exists a measurable function $g:\Omega\to \R^I$ such that $f\pm g\in \mathcal{S}_{\Gamma}^1$ and $g(\omega)\ne \bold{0}$ whenever $f(\omega)\not\in \mathrm{ex}\,\Gamma(\omega)$. 
\end{enumerate}
\end{lem}

The following result is a further extension of \cite[Theorem 4.3]{ks14} to the lcHs setting.  

\begin{thm}
\label{thm7}
If $X_i$ is a quasicomplete lcHs for each $i\in I$ and $m:\F\to \prod_{i\in I}X_i$ is a closed vector measure such that $\mathrm{dens}\,\prod_{i\in I}X_i<\kappa(\widehat{\F_E})$ for every $m$-\hspace{0pt}nonnull $E\in \F$, then $m$ satisfies the bang-\hspace{0pt}bang principle for every $\Gamma\in \M_{\mathit{bfc}}(\Omega,\R^I)$.
\end{thm}

\begin{proof}
It follows from the boundedness of $\Gamma$ that $\mathcal{S}_\Gamma^1$ is a bounded closed subset of $\prod_{i\in I}L^\infty(m_i)$, and hence, it is a weakly compact subset of $\prod_{i\in I}L^1(m_i)$ by Lemma \ref{lem2}. Define the integration operator $T:\prod_{i\in I}L^1(m_i)\to \prod_{i\in I}X_i$ by $T((f_i)_{i\in I})=(m_i(f_i))_{i\in I}$ for $f_i\in L^1(m_i)$ with $i\in I$. Then $T$ is a linear operator that is continuous in the weak topologies for $\prod_{i\in I}L^1(m_i)$ and $\prod_{i\in I}X_i$ since the integration operator $f_i\mapsto m_i(f_i)$ is continuous in the weak topologies for $L^1(m_i)$ and $X_i$ by Lemma \ref{lem1}. Let $\hat{x}\in T(\mathcal{S}_\Gamma^1)=\int \Gamma dm$ be given arbitrarily. Then the set $T^{-1}(\hat{x})\cap \mathcal{S}_\Gamma^1$ is a weakly compact, convex subset of $\prod_{i\in I}L^1(m_i)$, and hence, it has an extreme point $\hat{f}$. It suffices to show that $\hat{f}\in \mathcal{S}_{\mathrm{ex}\,\Gamma}^1$. If $\hat{f}\not\in \mathcal{S}_{\mathrm{ex}\,\Gamma}^1$, then there exists an $m$-\hspace{0pt}nonnull set $E\in \F$ such that $\hat{f}(\omega)\not\in \mathrm{ex}\,\Gamma(\omega)$ for every $\omega\in E$. By the boundedness of $\Gamma$ and Lemma \ref{lem5}, there exists $g=(g_i)_{i\in I}\in \prod_{i\in I}L^1(m_i)$ such that $\hat{f}\pm g\in \mathcal{S}_{\Gamma}^1$ and $g\ne \bold{0}$ on $E$. Since $m$ is a Lyapunov measure, the range of $m$ is convex in $\prod_{i\in I}X_i$. Take $F\subset E$ with $m(F)=\frac{1}{2}m(E)$ and define $\hat{g}:\Omega\to X$ by
$$
\hat{g}(\omega)=
\begin{cases}
\hspace{0.35cm} g(\omega) & \text{if $\omega\in E\setminus F$}, \\
-g(\omega) & \text{if $\omega\in F$}, \\
\quad \bold{0} & \text{otherwise}. 
\end{cases}
$$
Then $\hat{f}=\frac{1}{2}(\hat{f}+\hat{g})+\frac{1}{2}(\hat{f}-\hat{g})$ and $\hat{f}\pm \hat{g}\in \mathcal{S}_{\Gamma}^1$. A simple calculation yields $\int \hat{g}dm=(m_i(\hat{g}_i))_{i\in I}=\bold{0}\in X$, and hence, $\hat{f}\pm \hat{g}\in T^{-1}(\hat{x})\cap \mathcal{S}_\Gamma^1$. This contradicts the fact that $\hat{f}$ is an extreme point of $T^{-1}(\hat{x})\cap \mathcal{S}_\Gamma^1$.
\end{proof}

\begin{lem}
\label{lem6}
Let $X_i$ be a lcHs and $m_i:\F\to X_i$ be a vector measure for each $i\in I$ and define the set by 
$$
\mathcal{I}_E=\left\{ (f_i)_{i\in I}\in \prod_{i\in I}L^\infty_E(m_i) \mid 0\le f_i\le 1\ \forall i\in I \right\}, \quad E\in \F. 
$$ 
We then have 
$$
\mathrm{ex}\,\mathcal{I}_E=\left\{ (\chi_{A_i})_{i\in I}\in \prod_{i\in I}L^\infty_E(m_i) \mid A_i\in \F_E \ \forall i\in I \right\}.
$$ 
\end{lem}

\begin{proof}
Clearly, any sequence of characteristic functions $(\chi_{A_i})_{i\in I}$ is an extreme point of $\I_E$. On the other hand, if an extreme point $(f_i)_{i\in I}$ in $\I_E$ is not a sequence of characteristic functions, then without loss of generality, we may assume that there exist $\varepsilon>0$ and $m_j$-\hspace{0pt}nonnull set $A\in \F_E$ for some index $j\in I$ such that $\varepsilon\le f_j\le 1-\varepsilon$ on $A$. Then the sequence of functions $[f_j\pm \varepsilon \chi_A,(f_i)_{i\in I\setminus \{ j \}}]$ belongs to $\I_E$ and
$$
(f_i)_{i\in I}=\frac{1}{2}\left[ f_j+\varepsilon \chi_A,(f_i)_{i\in I\setminus \{ j \}} \right]+\frac{1}{2}\left[ f_j-\varepsilon \chi_A,(f_i)_{i\in I\setminus \{ j \}} \right],
$$
a contradiction to the fact that $(f_i)_{i\in I}$ is an extreme point of $\I_E$. 
\end{proof}

\begin{thm}
\label{thm6}
Let $X_i$ be a quasicomplete separable lcHs with a vector subspace that is isomorphic to an infinite-\hspace{0pt}dimensional Banach space for each $i\in I$ and $m:\F\to \prod_{i\in I}X_i$ is a closed vector measure. Then $m$ is saturated if and only if every vector measure in $\mathit{ca}(\Omega,m,\prod_{i\in I}X_i)$ satisfies the bang-\hspace{0pt}bang principle for every $\Gamma\in \M_\textit{bfc}(\Omega,\R^I)$.
\end{thm}

\begin{proof}
Suppose that every $n\in \mathit{ca}(\Omega,m,\prod_{i\in I}X_i)$ satisfies the bang-\hspace{0pt}bang principle for every $\Gamma\in \M_\textit{bfc}(\Omega,\R^I)$. Take any $E\in \F$ and define the multifunction $\Gamma_E:\Omega\twoheadrightarrow \R^I$ by
$$
\Gamma_E(t)=\{ (f_i(t))_{i\in I}\in \R^I \mid (f_i)_{i\in I}\in \I_E \}.
$$
We then have $\Gamma_E\in \M_\textit{bfc}(\Omega,\R^I)$ and $\mathcal{S}^1_{\Gamma_E}=\I_E$. Since $\I_E$ is weakly compact in the product topology of $\prod_{i\in I}L^1(n_i)$ by Lemma \ref{lem2} and the integration operator $T:\prod_{i\in I}L^1(n_i)\to \prod_{i\in I}X_i$ provided in the proof of Theorem \ref{thm7} is continuous for the weak topologies in $\prod_{i\in I}L^1(n_i)$, the set $T(\mathcal{S}^1_{\Gamma_E})=\int \Gamma_Edn$ is weakly compact and convex in $X$. By Lemma \ref{lem5}, we have $\mathcal{S}^1_{\mathrm{ex}\,\Gamma_E}=\mathrm{ex}\,\mathcal{S}^1_{\Gamma_E}=\mathrm{ex}\,\I_E$, and hence, it follows from Lemma \ref{lem6} that
$$
\int \Gamma_E dn=\int \mathrm{ex}\,\Gamma_E dn=\left\{ (n_i(A_i))_{i\in I}\in X\mid A_i\in \F_E \ \forall i\in I \right\}=n(\F_E). 
$$
The bang-\hspace{0pt}bang principle for $n$ implies that $n(\F_E)$ is weakly compact and convex in $\prod_{i\in I}X_i$ for every $E\in \F$. Hence, every $n\in \mathit{ca}(\Omega,m,\prod_{i\in I}X_i)$ is a Lyapunov measure. The saturation of $m$ follows from Theorem \ref{thm8}. The converse implication is a consequence of Theorem \ref{thm7}. 
\end{proof}

\subsection{Lyapunov Control Systems in LcHs}
As an application of the bang-\hspace{0pt}bang principle, we examine Lyapunov control systems developed by \citeauthor{kk75} \cite{kl73,kk73,kk75,kk78,kn75,kn76a,kn77}. 

Let $X$ be a lcHs and $X^I$ be the product space $\prod_{i\in I}X_i$ with $X_i=X$ for each $i\in I$. A vector measure $m:\F\to X^I$ is a \textit{control system} in $X^I$ if the sum $\sum_{i\in I}x_i$ is (unconditionally) convergent in $X$ for every $x_i\in m_i(\F)$ with $i\in I$. The attainable set for a control system $m:\F\to X^I$ and a measurable correspondence $\Gamma:\Omega\twoheadrightarrow \R^I$ is defined by
$$
\A_\Gamma(m)=\left\{ \sum_{i\in I}m_i(f_i)\in X\mid (f_i)_{i\in I}\in \mathcal{S}^1_\Gamma, \text{ $\sum_{i\in I}m_i(f_i)$ exists} \right\}.
$$

\begin{df}
A vector measure $m:\F\to X^I$ is a \textit{Lyapunov control system} for a multifunction $\Gamma:\Omega\twoheadrightarrow \R^I$ if $\A_\Gamma(m)=\A_{\mathrm{ex}\,\Gamma}(m)$. 
\end{df}

We provide a necessary and sufficient condition for control systems to be Lyapunov in terms of the saturation property. 

\begin{thm}
Let $X$ be a quasicomplete separable lcHs and $m:\F\to X^I$ be a closed vector measure. Then $m$ is saturated if and only if every vector measure in $\mathit{ca}(\Omega,m,X^I)$ is a Lyapunov control system for every $\Gamma\in \M_{\textit{bfc}}(\Omega,\R^I)$. 
\end{thm}

\begin{proof}
Suppose that every $n\in \mathit{ca}(\Omega,m,X^I)$ is a Lyapunov control system for every $\Gamma\in \M_{\textit{bfc}}(\Omega,\R^I)$. Let $\Gamma^j_E\in \M_\textit{bfc}(\Omega,\R^I)$ be the multifunction defined by 
$$
\Gamma^j_E(\omega)=\{ (f(\omega),(0)_{i\in I\setminus \{ j \}})\in \R^I \mid 0\le f\le 1,\,f\in L^\infty_E(n_j) \},
$$
where the nonzero entry $f(\omega)$ of the defining sequence in $\Gamma^j_E(\omega)$ appears in the $j$th component in $\R^I$. In view of Lemmas \ref{lem5} and \ref{lem6}, we have
$$
\mathcal{S}_{\mathrm{ex}\,\Gamma^j_E}^1=\mathrm{ex}\,\mathcal{S}_{\Gamma^j_E}^1=\{ (\chi_A,(0)_{i\in I\setminus \{ j \}})\mid A\in \F_E \}.
$$
Since $n$ is a Lyapunov control system for $\Gamma^j_E$, we obtain
$$
\A_{\Gamma^j_{E}}(n)=\A_{\mathrm{ex}\,\Gamma^j_E}(n)=\{ n_j(A)\in X\mid A\in \F_E \}=n_j(\F_E).
$$
By Lemma \ref{lem2}, $\mathcal{S}_{\Gamma^j_E}^1$ is weakly compact in $\prod_{i\in I}L^1(n_i)$ and by Lemma \ref{lem1}, $\A_{\Gamma^j_{E}}(n)$ is weakly compact and convex in $X$. Therefore, $n_j(\F_E)$ is weakly compact and convex in $X$ for every $E\in \F$, and hence, $n$ is a Lyapunov measure. The saturation of $n$ follows from Theorem \ref{thm8}. The converse implication is a consequence of Theorem \ref{thm7}. 
\end{proof}

\end{document}